\documentclass{gtart}

\usepackage{amsmath,amssymb,bm}
\usepackage{graphicx}
\usepackage{url}
\usepackage{cite}
\usepackage{enumerate}      

\usepackage[inner=30mm, outer=30mm, textheight=225mm]{geometry}

\def\C{\mathbb{C}}

\def\R{\mathbb{R}}

\def\PSL{PSL_2(\C)}

\def\tri{\mathcal{T}}
\def\para{\mathcal{P}}
\def\hyperboloid{H}
\def\projectivedisc{D}

\DeclareMathOperator{\Isom}{Isom}

\DeclareMathOperator{\vol}{Vol}

\theoremstyle{plain}

\newtheorem{theorem}{Theorem}[section]
\newtheorem{lemma}[theorem]{Lemma}
\newtheorem{proposition}[theorem]{Proposition}
\newtheorem{corollary}[theorem]{Corollary}

\newtheorem{definition}[theorem]{Definition}
\newtheorem{remark}[theorem]{Remark}

\begin{document}

\title{{Thurston's Spinning Construction and Solutions to the\\ Hyperbolic Gluing Equations for Closed Hyperbolic 3-Manifolds}}
\author{Feng Luo, Stephan Tillmann and Tian Yang}

\begin{abstract}
We show that the hyperbolic structure on a closed, orientable, hyperbolic $3$-manifold can be constructed from a solution to the hyperbolic gluing equations using any triangulation with essential edges. The key ingredients in the proof are Thurston's spinning construction and a volume rigidity result attributed by Dunfield to Thurston, Gromov and Goldman. As an application, we show that this gives a new algorithm to detect hyperbolic structures on closed 3--manifolds.
\end{abstract}

\maketitle


\section{Introduction}

In his Princeton notes \cite{Th1}, Thurston introduced a system of algebraic equations
---called the hyperbolic gluing equations--- for constructing hyperbolic
metrics on orientable $3$--manifolds with torus cusps. He used solutions to
the hyperbolic gluing equations to produce a complete hyperbolic metric on the
figure-eight knot complement in the early stages of formulating his
geometrization conjecture. On a \emph{closed,} triangulated, oriented 3--manifold
$M,$ the hyperbolic gluing equations can be defined
in the same way: We assign to each edge of
each oriented tetrahedron in the triangulation a shape parameter
$z\in \mathbb{C}\setminus \{0,1\}$, such that
\begin{enumerate}[(a)]
\item opposite edges of each tetrahedron have the same shape parameter;
\item the three shape parameters assigned to the three pairs of opposite
edges in each tetrahedron are $z,$ $\frac{1}{1-z}$ and
$\frac{z-1}{z}$ subject to an orientation convention; and
\item for each edge $e$ in $M,$ if $z_1,...,z_k$ are the
shape parameters assigned to the tetrahedron edges identified with $e,$ then
we have
\begin{equation}\label{monodromy}
\prod_{i=1}^kz_i=1.
\end{equation}
\end{enumerate}
The equations (\ref{monodromy}) are termed the \emph{hyperbolic gluing equations}, and the set of all solutions is the \emph{parameter space} $\para(M).$ The space $\para(M)$ depends on the triangulation of $M.$ We will study the solutions using a volume function and representations of the fundamental group. In order to to produce the representations, we need the following topological hypothesis.

The edge $e$ in $M$ is termed \emph{essential} if it is not a null-homotopic loop in $M.$ This is clearly the case if it has distinct end-points, but we allow the triangulation of $M$ to be semi-simplicial (or singular), so that some or all edges may be loops in $M.$ We will always assume that all edges are essential. In this case, to each solution $Z \in \para(M),$ there is an associated representation $\rho_{Z}\co \pi_1(M)\to \PSL.$

Given $Z \in \para(M),$ the \emph{volume of $Z,$} denoted $\vol(Z),$ is defined using the Lobachevsky-Milnor formula \cite{Mi} as the sum of the signed volumes of the oriented ideal hyperbolic tetrahedra specified by $Z.$

Suppose that $M$ is hyperbolic and denote $\vol(M)$ the hyperbolic volume of $M.$ Using results about representation volume, we show that for each $Z \in \para(M),$ $\vol(Z) \in [-\vol(M), \vol(M)],$ and it is a consequence of Stokes' theorem that $\vol\co \para(M) \to \R$ is constant on the topological components of $\para(M).$

Our main observation is that, just as in the case of hyperbolic manifolds with cusps, we can recover the hyperbolic structure on a closed hyperbolic $3$-manifold from any solution of the hyperbolic gluing equations with maximal volume:

\begin{theorem}\label{thm:main}
Let $M$ be a closed, oriented, triangulated, hyperbolic $3$--manifold with the property that all edges in $M$ are essential. Then there exists $Z_\infty \in \para(M)$ such that $\vol(Z_{\infty})=\vol(M).$ Moreover, for any such $Z_\infty,$ the associated holonomy representation $\rho_{\infty}\co\pi_1(M)\rightarrow\PSL$ is discrete and faithful, and $M$ is isometric with $\mathbb{H}^3/\rho_{\infty}(\pi_1(M)).$
\end{theorem}

The existence statement in the proof of Theorem~\ref{thm:main} is constructive and makes crucial use of Thurston's spinning construction from \cite{Th2}. The remainder is proved using a volume rigidity result attributed by Dunfield to Thurston, Gromov and Goldman (see \cite{Du}, Theorem 6.1).

The only known algorithmic detection and description of hyperbolic structures on closed 3--manifolds is due to Casson and Manning \cite{Ma} and uses the automatic structure of the fundamental group. As an application, we show that the above results allow the algorithmic construction of the hyperbolic structure without having to find a solution to the word problem and, moreover, we show that the recognition of small Seifert fibred spaces is also within the scope of our approach:

\begin{theorem}[Casson-Manning\cite{Ma} \& Rubinstein \cite{Ru}]\label{thm:app}
There exists an algorithm, which will, given a triangulated, closed, orientable 3--manifold $M$ decide whether or not it has a hyperbolic structure or a small Seifert fibred structure. Moreover, if $M$ has a hyperbolic structure or a small Seifert fibred structure, then the structure is algorithmically constructible.
\end{theorem}

This paper is organised as follows. In Section~\ref{sec:straightening}, we summarise some background material on hyperbolic geometry and prove
a technical lemma which will be used in the spinning construction.
Thurston's spinning construction is reviewed in Section~\ref{sec:spinning}, and basic properties of the parameter space are given in Section~\ref{sec:parameter space}. The proofs of Theorems~\ref{thm:main} and~\ref{thm:app} are given in Section~\ref{sec:proofs}.


\section{The straightening map for hyperbolic simplices}
\label{sec:straightening}

In this section, we will prove a technical result (Lemma \ref{tech})
that will be used in the subsequent sections. The proof of Lemma
\ref{tech} makes use of both the hyperboloid and Klein models of
$n$-dimensional hyperbolic space which are now briefly reviewed.


\subsection{The hyperboloid and Klein models of $\mathbb{H}^n$}

Let $\mathbb{E}^{n,1}$ be the \textit{Minkowski space} which is
$\mathbb{R}^{n+1}$ with inner product $<,>$ defined by
\begin{equation*}
\langle x,y\rangle=\sum_{i=1}^nx_iy_i-x_{n+1}y_{n+1},
\end{equation*}
where $x=(x_1,...,x_{n+1})$ and $y=(y_1,...,y_{n+1}).$ The
\textit{hyperboloid model} $\hyperboloid^n$ of $\mathbb{H}^n$ is
\begin{equation*}
\hyperboloid^n=\{x\in \mathbb{E}^{n,1}\ |\ \langle x,x \rangle=-1,\ x_{n+1}>0\},
\end{equation*}
and the \textit{light cone} $\mathbb{L}$ is defined to be
\begin{equation*}
\mathbb{L}=\{x\in \mathbb{E}^{n,1}\ |\ \langle x,x \rangle=0,\ x_{n+1}>0\}.
\end{equation*}
Geodesics in the hyperboloid model $\hyperboloid^n$ are of the form
$\hyperboloid^n\cap L,$ where $L$ is a $2$-dimensional linear
subspace in $\mathbb{E}^{n,1}.$

The open unit disc $\projectivedisc^n\subset\mathbb{R}^n$ is identified with a subset of $\mathbb{E}^{n,1}$ via the natural maps
$$\projectivedisc^n\subset\mathbb{R}^n \hookrightarrow \mathbb{R}^n\times\{1\}\subset\mathbb{E}^{n,1}.$$
The radial projection $\pi\co\hyperboloid^n\rightarrow\projectivedisc^n,$ defined by 
$\pi(x_1,...,x_{n+1})=\frac{1}{x_{n+1}}(x_1,...,x_{n+1}),$
is
a bijection, and the \textit{Klein model} of the $n$-dimensional
hyperbolic space is $\projectivedisc^n$ with the Riemannian metric
induced from $\hyperboloid^n$ via $\pi.$ Since $\pi$ maps geodesics in
$\hyperboloid^n$ to line segments, geodesics in the Klein model
$\projectivedisc^n$ are Euclidean straight lines. The boundary
$\partial\projectivedisc^n$ of the Klein model $\projectivedisc^n$ is called
the \textit{sphere at infinity}, denoted by $S_{\infty}^{n-1}$, and
$\overline{\projectivedisc^n}=\projectivedisc^n\cup S_{\infty}^{n-1}$ is
called the \textit{compactification} of $\projectivedisc^n$. We see that
$\overline{\projectivedisc^n}$ naturally inherits a smooth structure from
$\mathbb{R}^n\times\{1\}\subset\mathbb{E}^{n,1}$.

We will denote
\begin{equation*}
\mathbb{E}^{n,1}_{<0}=\{x\in \mathbb{E}^{n,1}\ |\ \langle x,x \rangle <0,\ x_{n+1}>0\}
\end{equation*}
and
\begin{equation*}
\mathbb{E}^{n,1}_{\le0}=\{x\in \mathbb{E}^{n,1}\ |\ \langle x,x \rangle \le 0,\ x_{n+1}>0\}.
\end{equation*}
We have $H^n \subset \mathbb{E}^{n,1}_{<0} \subset \mathbb{E}^{n,1}_{\le0}$ and the radial projection
$\pi\co\hyperboloid^n\rightarrow\projectivedisc^n$ can be viewed as
the restriction to $\hyperboloid^n$ of the radial projection
$\pi\co\mathbb{E}^{n,1}_{\le0}\rightarrow\overline{\projectivedisc^n}.$
\\


\subsection{The straightening map for hyperbolic simplices}

Let
$$\Delta^k=\{(t_0,...,t_k)\in\mathbb{R}^{k+1}\ |\ t_i\geqslant0,\
i=0,...,k\ \text{and}\ \sum t_i=1\}$$ be the standard $k$-simplex
with vertex set $\{e_i\}.$ Following \cite{Th2}, any $k+1$ points
$v_0,...,v_k,$ $1\leqslant k\leqslant n,$ in $\projectivedisc^n$
determine a \textit{straightening map} (or \textit{straight
$k$-simplex})
$\sigma_{v_0,...,v_k}\co\Delta^k\rightarrow\projectivedisc^n,$ whose
image is the convex hull of $v_0,...,v_k.$ It is defined as follows.
In the Minkowski space $\mathbb{E}^{n,1}$, the $k+1$ points
$u_i=\pi^{-1}(v_i) \in \hyperboloid^n,$ $i \in \{0,\ldots, k\},$
determine an affine $k$-simplex
$\sigma\co\Delta^k\rightarrow\mathbb{E}^{n,1}_{<0}$ by the
barycentric coordinates
\begin{equation*}
\sigma(t_0,...,t_k)=\sum_{i=0}^kt_iu_i.
\end{equation*}

\begin{definition}
The straightening map
$\sigma_{v_0,...,v_k}\co\Delta^k\rightarrow\projectivedisc^n$ is defined
by $\sigma_{v_0,...,v_k}=\pi\circ\sigma,$ and
$\sigma\co\Delta^k\rightarrow\mathbb{E}^{n,1}_{<0}$ is called the affine
simplex of $\sigma_{v_0,...,v_k}.$
\end{definition}

A \emph{straight ideal simplex} is defined similarly. In this case, the vertices $v_0,...,v_k$ are in $\partial\projectivedisc^n=S_{\infty}^{n-1}.$ Choose $u_0,...,u_k$ in the light cone $\mathbb{L}$ so that $\pi(u_i)=v_i$ for all $i$. Let the affine simplex $\sigma\co\Delta^k\to \mathbb{E}^{n,1}_{\le0}$ be defined by
\begin{equation*}
\sigma(t_0,...,t_k)=\sum_{i=0}^kt_iu_i.
\end{equation*}
Then a straight ideal simplex with vertices $v_0,...,v_k$ is defined to be 
$$\pi\circ\sigma\co\Delta^k\rightarrow\overline{\projectivedisc^n}.$$ 
Note that, unlike the compact case, the points $u_i$ are not unique since $\pi^{-1}(v_i)=\{t\cdot u_i\ |\ t>0\}.$ Thus, a straight ideal simplex is not uniquely determined by its set of vertices $\{v_0,...,v_k\}.$ Instead, given $\{v_0,...,v_k\}$ in $\partial\projectivedisc^n,$ the set of all straight ideal simplices with vertex set $\{v_0,...,v_k\}$ is bijective to $\mathbb{R}_{>0}^{k+1}.$ In Penner's language \cite{penner}, a straight ideal simplex is a decorated hyperbolic ideal simplex with a positive number assigned to each vertex. If a subsimplex of a straight ideal simplex is viewed as a simplex in its own right, it will always inherit the same decoration unless stated otherwise.

\begin{proposition}\label{straight}
The straight simplex (and straight ideal simplex) is natural in the
following sense.
\begin{enumerate}
\item If $\Delta'$ is an $m$-face of $\Delta^k$ so that
$\sigma_{v_0,...,v_k}(\Delta')$ has vertices $v_{i_0},...,v_{i_m},$
then
$$\sigma_{v_0,...,v_k}|_{\Delta'}=\sigma_{v_{i_0},...,v_{i_m}}.$$

\item If $g\in \Isom(\mathbb{H}^n),$ the group of isometries of
$\mathbb{H}^n,$ then
$$g\circ\sigma_{v_0,...,v_k}=\sigma_{g\cdot v_0,...,g\cdot v_k}.$$
\end{enumerate}
\end{proposition}

\begin{proof}
The first part is true since both sides of the equation equal
$\pi\circ\sigma|_{\Delta'}.$ To see the second part, if $\sigma$ is
the affine simplex of $\sigma_{v_0,...,v_k},$ then $g\cdot\sigma$ is
the affine simplex of $\sigma_{g\cdot v_0,...,g\cdot v_k}.$
Therefore, $\pi\circ g\cdot\sigma=\sigma_{g\cdot v_0,...,g\cdot
v_k}.$ Since the radial projection commutes with each $g\in
\Isom(\mathbb{H}^n),$
$g\cdot\sigma_{v_0,...,v_k}=g\cdot(\pi\circ\sigma)=\pi\circ
g\cdot\sigma=\sigma_{g\cdot v_0,...,g\cdot v_k}.$
\end{proof}

\begin{lemma}\label{tech}
Let $\{\sigma_t\co\Delta^k\rightarrow\projectivedisc^n \mid t \in
\mathbb{R}_{\geqslant0} \}$ be a family of straight $k$-simplices so
that the $i$-th vertex $v_{i,t}$ of $\sigma_t$ lies on the geodesic
ray $l_i$ and $v_{i,t}$ moves toward the end point $v_i^*$ of $l_i$
at unit speed, i.e., $d(v_{i,0},v_{i,t})=t$. If $v_0^*,...,v_k^*$
are pairwise distinct, then as $t$ tends to $\infty$ the family
$\{\sigma_t\}$ converges pointwise to an ideal straight $k$-simplex
$\sigma_{\infty}\co\Delta^k\rightarrow\projectivedisc^n$ whose
vertices are $v_0^*,...,v_k^*$. Furthermore,
$\lim_{t\rightarrow\infty}\vol(\sigma_t)=\vol(\sigma_{\infty}).$
\end{lemma}

\begin{proof}
The proof is based on the following observation. Suppose
$\gamma(t)=(\gamma_1(t),...,\gamma_{n+1}(t)),$ for $t\in\mathbb{R}$ is a
geodesic in the hyperboloid model $H^n,$ which is parameterised by unit speed. Then
\begin{equation}\label{lim}
\lim_{t\rightarrow\infty}\frac{\gamma_{n+1}(t)}{\cosh(t)}
\end{equation}
exists and is in $\mathbb{R}_{>0}$.

Indeed, the result is obvious for the specific geodesic
$\delta(t)=(\sinh(t),0,...,0,\cosh(t))$. Now any other geodesic
$\gamma(t)$ is the image of $\delta(t)$ under a linear
transformation $A$ (independent of $t$) preserving $H^n$.

To prove the proposition, for $\sigma_t$, let $\widetilde{\sigma_t}$
be the affine simplex in $\mathbb{E}^{n,1}$ with the same set of
vertices as that of $\sigma_t$. Then by definition
$\sigma_t=\pi\circ\widetilde{\sigma_t}$. Let
$\tau_t=\frac{1}{\cosh(t)}\cdot\widetilde{\sigma_t}$ be a new
affine simplex. By the choice of vertices of $\widetilde{\sigma_t}$ and
(\ref{lim}), the family of maps $\tau_t$ converges pointwise to an
affine map $\tau_{\infty}\co\Delta^k\rightarrow\mathbb{E}^{n,1}$. It
follows that $\sigma_t=\pi(\widetilde{\sigma_t})=\pi(\tau_t)$
converges pointwise to an ideal straight $k$-simplex.

To see the convergence of the volume of $\sigma_t$ to
$\sigma_{\infty}$, we observe that since the vertices converge, the
corresponding dihedral angles of $\sigma_t$ converge to those of
$\sigma_{\infty}$. Now by the solution of Milnor's conjecture on the
volume of hyperbolic simplices \cite{Lu1}, we conclude that the
volumes converge.
\end{proof}


\section{Thurston's spinning construction}
\label{sec:spinning}

Let $M$ be a closed hyperbolic $3$-manifold, and $\tri$ be a (possibly semi-simplicial) triangulation of $M$ with the property that all edges are essential. We denote $M^{(k)}$ the $k$--skeleton in $M.$
Let $p\co\widetilde{M}\rightarrow M$ be the
universal cover of $M$ and $\widetilde{\tri}$ be the triangulation of $\widetilde{M}$ induced by $\tri.$ Since all edges in $M$ are essential, it follows that if $\sigma$ is a 3--simplex in $M,$ then any lift
$\widetilde{\sigma}$ of $\sigma$ to $\widetilde{M}$ has four distinct vertices.
We may identify $\widetilde{M}$ with the Klein model $\projectivedisc^3$ using the hyperbolic metric, and the natural action of $\pi_1(M)$ by deck transformations on $\projectivedisc^3$ is by isometries. Thurston's spinning construction in our context is summarized as follows.

\begin{proposition}\label{pro:spin}
Let M be a closed, triangulated, hyperbolic 3-manifold with the
property that all edges in M are essential. Then there exists a
continuous family of piecewise smooth, $\pi_1(M)$--equivariant maps
$$\widetilde{F}_t\co\widetilde{M}\rightarrow\projectivedisc^3,$$
$t\in[0,\infty),$ and a piecewise smooth, $\pi_1(M)$--equivariant
map
$\widetilde{F}_{\infty}\co\widetilde{M}\rightarrow\overline{\projectivedisc^3},$
such that
\begin{enumerate}
\item for each vertex $\widetilde{v}$ of $\widetilde{\tri},$
the vertex $\widetilde{F}_t(\widetilde{v})$ approaches
$S_{\infty}^2,$ and for each 3-simplex $\widetilde{\sigma}$ of
$\widetilde{\tri},$ $\widetilde{F}_t(\widetilde{\sigma})$
is a hyperbolic tetrahedron,
\item $\widetilde{F}_t$ descents to a piecewise smooth map
$F_t\co M\rightarrow M$ which is homotopic to the identity map of
$M,$ for all $t\in[0,\infty),$
\item
$\widetilde{F}_\infty=\lim_{t\rightarrow\infty}\widetilde{F}_t$
pointwise in $\overline{\projectivedisc^3},$
\item for every 3--simplex $\widetilde{\sigma}$ in
$\widetilde{\tri},$ $\widetilde{F}_\infty(\widetilde{\sigma})$ is an
ideal tetrahedron with $4$ distinct vertices in
$\overline{\projectivedisc^3},$
\item $\widetilde{F}_\infty(\widetilde{M}\setminus p^{-1}(M^{(0)}))\subset\projectivedisc^3,$
and $\widetilde{F}_\infty|_{\widetilde{M}\setminus p^{-1}(M^{(0)})}$
descends to a piecewise smooth map $F_\infty\co M\setminus
M^{(0)}\rightarrow M.$
\end{enumerate}
\end{proposition}

\begin{proof}

Suppose $M^{(0)}=\{v_i\}$, and for each $v_i$, let $l_i$ be a (not
necessarily closed or simple) geodesic passing through $v_i$, and
denote $L=\{l_i\}$. Suppose that for every $3$-simplex
$\widetilde{\sigma}$ in $\widetilde{M}$ any two lifts of geodesics
in $L$ which pass through different vertices of $\widetilde{\sigma}$
have no endpoints in common. We parameterize
$l_i\co(-\infty,\infty)\rightarrow M$ by $l_i(0)=v_i$ and
$\|l_i'(t)\|_{\mathbb{H}^3}=1$ for all $i\in\{1,...,|M^{(0)}|\}$ and
$t\in\mathbb{R}.$ To construct
$\widetilde{F}_t\co\widetilde{M}\rightarrow\projectivedisc^3,$ for
each $i\in\{1,...,|M^{(0)}|\}$, pick $\widetilde{v_i}\in
p^{-1}(v_i)$ and a lift $\widetilde{l_i}$ of $l_i$ passing through
$\widetilde{v_i},$ and define
$$\widetilde{F}_t(\widetilde{v_i})=\exp_{\widetilde{v_i}}(t\cdot \widetilde{l_i}'(0)),$$
where $\exp_v$ is the exponential map at $v$.

We define $\widetilde{F}_t\co
p^{-1}(M^{(0)})\rightarrow\projectivedisc^3$ by
\begin{equation}
\widetilde{F}_t(\gamma\cdot\widetilde{v})=\gamma\cdot
\widetilde{F}_t(\widetilde{v}),\ \ \widetilde{v}\in p^{-1}(M^{(0)})
\text{ and }\gamma \in \pi_1(M),
\end{equation}
using the action of $\pi_1(M)$ on $\widetilde{M}$ and $\projectivedisc^3$ by deck transformations.

The map $\widetilde{F}_t$ is now extended to the $3$--simplices in
$\widetilde{\tri}$ by straightening maps. By part (1) of Propositon
\ref{straight}, this gives a well-defined map $\widetilde{F}_t\co
\widetilde{M} \to \projectivedisc^3$ since the straightening maps
agree on intersections of 3--simplices in $\widetilde{M}.$ By part
(2) of Proposition \ref{straight}, the map $\widetilde{F}_t$ is
$\pi_1(M)$--equivariant.

By equivariance, $\widetilde{F}_t$ descends to a piecewise smooth
map
$$F_t\co M\rightarrow M.$$
For any $t_0\in\mathbb{R}_+,$ the map
$H_{t_0}\co M\times[0,1]\rightarrow M$ defined by
$$H_{t_0}(x,t)=F_{t_0t}(x)$$
provides a homotopy between $F_0$ and $F_{t_0}.$

Take the homotopy between $id_M$ and $F_0$ to be the \emph{straight
line homotopy}. Namely, for $\widetilde{x}\in
\widetilde{M}\cong\projectivedisc^3,$ there is a unique geodesic
segment $l_{\widetilde{x}}\co[0,1]\rightarrow\projectivedisc^3$ such that
$l_{\widetilde{x}}(0)=\widetilde{x}$ and $l_{\widetilde{x}}(1)=\widetilde{F}_0(\widetilde{x}).$ The map $H_0\co
M\times[0,1]\rightarrow M$ defined by
$$H_0(x,t)=p\circ l_{\widetilde{x}}(t),$$
where $\widetilde{x}\in\widetilde{M}$ is any lift of $x$, provides a homotopy from $id_M$ to $F_0.$ This proves parts (1) and
(2).

We now define a $\pi_1(M)$-equivariant map
$$\widetilde{F}_{\infty}\co p^{-1}(M^{(0)})\rightarrow\overline{\projectivedisc^3}$$
as follows. For each $v \in M^{(0)},$ choose a lift
$\widetilde{v}\in p^{-1}(M^{(0)})$ and let $\widetilde{l}$ be the
corresponding lift of the associated element of $L.$ As $t$ tends to
infinity, $F_t(\widetilde{v})=\exp_{\widetilde{v}}(t\cdot
\widetilde{l}'(0))$ approaches an end-point
$\widetilde{v}^*\in S_{\infty}^2$ of $\widetilde{l}.$ Define
$\widetilde{F}_{\infty}(\widetilde{v})$ to be $\widetilde{v^*}$ and
$$\widetilde{F}_{\infty}(\gamma\cdot\widetilde{v})=\gamma\cdot\widetilde{v}^*,\ \ \gamma\in\pi_1(M).$$
By Lemma \ref{tech}, on each simplex $\widetilde{\sigma}$ in
$\widetilde{\tri},$ $\{\widetilde{F}_t|_{\widetilde{\sigma}}\}$
converges, so we can define
$$\widetilde{F}_{\infty}|_{\widetilde{\sigma}}\doteq\lim_{t\rightarrow\infty}\widetilde{F}_t|_{\widetilde{\sigma}},$$
and this agrees with the above definition on its vertices. By Part
(1) of Proposition \ref{straight}, $\widetilde{F}_{\infty}$ is
well-defined on intersections of 3--simplices in $\widetilde{M},$
and hence on $\widetilde{M}.$ Since each $\widetilde{F}_t$ is
$\pi_1(M)$-equivariant, so is $\widetilde{F}_{\infty}.$ This
completes the proof of part (3).

By the assumption on the end-points of the set of lifted geodesics, the
vertices of $\widetilde{F}_{\infty}(\widetilde{\sigma})$ are
distinct. Therefore $\widetilde{F}_{\infty}(\widetilde{\sigma})$ is
an ideal tetrahedron with four distinct vertices in
$\overline{\projectivedisc^3}.$ This proves (4), and (5) is a direct
consequence of (4).
\end{proof}

\begin{remark}
The hypothesis on $L$ avoids the situation in \cite{Th1}, where simplices of higher dimension may be mapped into the sphere at infinity. In the language of \cite{Th1}, our subcomplex at infinity consists merely of the 0--skeleton.
\end{remark}


\section{The parameter space}
\label{sec:parameter space}

Throughout this section, we suppose that $M$ is a triangulated, oriented, closed 3--manifold. We denote $\Sigma^k$ the set of all $k$--simplices of the triangulation in $M.$ As above, the triangulation may be semi-simplicial, so an element of $\Sigma^k$ may not be an embedded $k$--simplex in $M.$ Nevertheless, elements of $\Sigma^1$ will be termed \emph{edges} and elements of $\Sigma^3$ are termed \emph{tetrahedra.}


\subsection{The hyperbolic gluing equations}

Let $\Delta^3$ be the standard 3--simplex with a chosen orientation. Suppose the edges from one vertex of $\Delta^3$ are labeled by $e_1,$ $e_2$ and $e_3$ so that the opposite edges have the same labeling. Then the cyclic order of $e_1,$ $e_2$ and $e_3$ viewed from each vertex depends only on the orientation of the 3--simplex, i.e.\thinspace is independent of the choice of the vertices. It follows that, up to orientation preserving symmetries, there are two possible labelings, and we will fix one of these labelings.

Each pair of opposite edges corresponds to a normal isotopy class of quadrilaterals (\textit{normal quadrilateral} for short) in $\Delta^3.$ There is a natural cyclic order on the set of normal quadrilaterals induced by the cyclic order on the edges from a vertex, and this order is preserved by all orientation preserving symmetries of $\Delta^3.$

If $\sigma \in \Sigma^3,$ then there is an orientation preserving map $\Delta^3 \to \sigma$ taking the $k$--simplices in $\Delta^3$ to elements of $\Sigma^k,$ and which is a bijection between the sets of normal quadrilaterals. This map induces a cyclic order of the normal quadrilaterals in $\sigma,$ and we denote the corresponding 3--cycle $\tau_\sigma.$ It follows from the above remarks, that this order is independent of the choice of the map. We define
$$\tau = \prod_{\sigma \in \Sigma^3} \tau_\sigma.$$

Let $e\in \Sigma^1,$ and $q$ be a normal quadrilateral in $\sigma.$ The index $i(q,e)$ is the integer $0,$ $1$ or $2$ defined as follows: $i(q,e)=0$ if $e$ is not an edge of $\sigma,$ $i(q,e)=1$ if $e$ is the only edge in $\sigma$ facing $q,$ and $i(q,e)=2$ if $e$ are the two edges in $\sigma$ facing $q.$

\begin{definition}\label{ThGE}
The parameter space $\para(M)$ is the set of all points $Z=(z_q)\in(\mathbb{C}\setminus\{0,1\})^Q,$ where $Q$ is the set of all normal quadrilaterals, satisfying the following two conditions:
\begin{enumerate}
\item[(a)] for each edge $e$ in $M,$
\begin{equation}\label{eq:gluing}
\prod_{q\in Q} z_q^{i(q,e)}=1,
\end{equation}
\item[(b)] for each $q\in Q$
\begin{equation}\label{eq:parameter}
z_{\tau q}=\frac{1}{1-z_{q}}.
\end{equation}
\end{enumerate}
\end{definition}
Equation (\ref{eq:gluing}) is the \emph{hyperbolic gluing equation} of $e,$ and (\ref{eq:parameter}) is the \emph{parameter relation} of $q.$

Let $z_{\sigma}=(z_{q},z_{\tau q},z_{\tau^2q})$ be the triple of complex numbers assigned to the three normal quadrilaterals in the tetrahedron $\sigma.$ We will often write $Z \in \para(M)$ as $Z = (z_\sigma).$


\subsection{The (lack of) geometry of solutions}

Unlike in the case of cusped hyperbolic 3--manifolds, solutions to the hyperbolic gluing equations cannot be used to directly construct a hyperbolic metric on a closed 3--manifold, as highlighted by the following result.

\begin{proposition}
Let $M$ be a triangulated, closed, oriented 3--manifold,
and $Z\in\para(M).$ Then there is at least one $3$--simplex $\sigma\in\Sigma^3$
such that $z_{\sigma}\in-\overline{\mathbb{H}},$ the closure of the lower
half plane.
\end{proposition}

\begin{proof} Suppose for all $\sigma,$ $z_{\sigma}\in\mathbb{H},$ then by
taking arguments, we would have got an \emph{angle structure} on $M,$ which is an assignment of real numbers, termed angles, in the range $(0,\pi)$ to each edge of each $3$--simplex such that the sum of all angles at each vertex of each $3$--simplex is $\pi,$ and such that around each edge $e$ of $\tri$ the sum of angles is $2k_e\pi$ with $k_e \ge 1.$ This induces a combinatorial angle structure on the link of each vertex in $M,$ that is, a function $a \co \{ \text{all corners in the link} \} \to \R.$ Since all 2--cells in the induced triangulation of the vertex link are triangles and have angle sum $\pi,$ the total combinatorial area $A(a)$ of the angle structure is zero. The combinatorial curvature $K_v$ at vertex $v$ of the induced triangulation is $2 \pi$ minus the sum of angles at $v,$ which equals $2(1-k_e)\pi,$ where $e$ is the edge containing $v.$ Whence $K_v$ is non-positive. Since the link of each vertex in $M$ is a sphere, the combinatorial Gau\ss--Bonnet formula (see \cite{LT}, Proposition 13) now implies that
$$
4 \pi = 2 \pi \chi(S^2) = A(a) + \sum K_v \le 0,
$$
which gives a contradiction.
\end{proof}

The above result is sharp in the sense that there may not be any negatively oriented tetrahedra. For instance, the one-tetrahedron triangulation of $L(4,1)$ yields the unique solution $(-1, \frac{1}{2}, 2).$

We remark that taking the arguments of $Z\in\para(M)$ defines an \textit{$S^1$--angle structure} on the closed $3$-manifold $M,$ which is the counterpart of an angle structure on a $3$--manifold with torus cusps (see \cite{Lu} for results on $S^1$--angle structures).


\subsection{The volume of solutions}

Recall that the Lobachevsky function is defined by:
$$\Lambda(\alpha) = -\int_0^{\alpha}\ln |2\sin t|dt$$
for any $\alpha \in \R.$

\begin{definition}\label{vol}
The volume of $z_{\sigma}$ is defined to be the sum of the Lobachevsky
functions
\begin{equation*}
\vol(z_{\sigma})=\Lambda(\arg(z_{q})) + \Lambda(\arg(z_{\tau q})) + \Lambda(\arg(z_{\tau^2q})),
\end{equation*}
and the volume of $(z_{\sigma})=Z \in \para(M)$ is defined by
$$\vol(Z)=\sum_{\sigma\in \Sigma^3}\vol(z_{\sigma}).$$
\end{definition}

The hyperbolic gluing equations are defined over the integers, so $Z \in \para(M)$ implies $\overline{Z} \in \para(M),$ where the latter point is obtained by taking the complex conjugates of all coordinates. It follows from the properties of the Lobachevsky function (see \cite{Mi}) that $\vol(Z) = - \vol(\overline{Z}).$


\subsection{The shape parameters of an ideal tetrahedron}
Let $\sigma$ be an ideal geodesic tetrahedron in $\mathbb{H}^3$ with vertices $\{v_i\}\subset S_{\infty}^2,$ $i=1,...,4.$ The order $(v_1,...,v_4)$ determines an \textit{orientation} of $\sigma.$ We call $\sigma$ \textit{positive} if the orientation of $\sigma$ coincides with the orientation of $\mathbb{H}^3,$ \textit{negative} if the orientation of $\sigma$ differs from the one of $\mathbb{H}^3,$ and flat if $\sigma$ lies in a totally geodesic plane.

\begin{definition}\label{shape}
Let $e_{ij}$ be the edge from $v_i$ to $v_j,$ and identify
$S_{\infty}^2$ with $\mathbb{C}\cup\{\infty\},$ then the shape
parameter of $\sigma$ at $e_{ij}$ (or edge invariant at $e_{ij}$) is defined by the cross-ratio

\begin{equation*}
\begin{split}
z_{ij}\doteq& (v_i,v_j;v_k,v_l)\\
=&\frac{v_i-v_k}{v_i-v_l}\cdot\frac{v_j-v_l}{v_j-v_k}\\
\end{split}
\end{equation*}

where $(i,j,k,l)$ is an even permutation of $(1,2,3,4).$
\end{definition}

A direct cross-ratio calculation shows the following well known:

\begin{proposition}\label{wellknown}
With the above notation:
\begin{enumerate}
\item[(1)] For all $\{i,j\} \cup \{k,l\} = \{1,2,3,4\},$
$$z_{ij}=z_{kl},\ \,$$
so opposite edges share the same shape parameter, and we can denote the shape parameter of $\sigma$ at $e_{ij}$ and $e_{kl}$ by $z_q,$ where $q$ is the normal quadrilateral facing $e_{ij}$ and $e_{kl},$ and
\item[(2)] if the 3--cycle $\tau_\sigma$ determines the cyclic order of the normal quadrilaterals in $\sigma,$ then
$$z_{\tau q}=\frac{1}{1-z_{q}}.$$
\end{enumerate}
\end{proposition}

For an ideal tetrahedron $\sigma$ with shape parameters $z_{q},z_{\tau q}$ and $z_{\tau^2q},$ the hyperbolic volume is calculated by Milnor as
\begin{equation}\label{Milnor}
\vol(\sigma)=\Lambda(\arg(z_{q})) + \Lambda(\arg(z_{\tau q})) + \Lambda(\arg(z_{\tau^2q})).
\end{equation}
Therefore, when $\sigma$ is positive, its shape parameters are in the upper half plane and $\vol(\sigma)>0;$ when $\sigma$ is flat, its shape parameters are real and $\vol(\sigma)=0$; and when $\sigma$ is negative, its shape parameters are in the lower half plane and $\vol(\sigma)<0.$


\subsection{The associated representation}
\label{subsec:Yoshida}

The following is essentially the construction described by Yoshida for cusped 3--manifolds in \cite{Yo}, \S 5, though we will take more care of details which are needed for our application.

We assume that the triangulation of $M$ consists of a pairwise disjoint union of standard 3--simplices, $\widetilde{\Delta} = \cup_{k=1}^{n} \widetilde{\Delta}_k,$ together with a collection $\Phi$ of Euclidean isometries between standard 2--simplices in $\widetilde{\Delta};$ termed \emph{face pairings}. Then $M = \widetilde{\Delta} / \Phi.$ Since $M$ is oriented, we may assume that all 3--simplices in $\widetilde{\Delta}$ are coherently oriented, so that each face pairing is orientation reversing.

Denote $p \co \widetilde{M} \to M$ the universal cover of $M.$ Lift the triangulation of $M$ to a $\pi_1(M)$--equivariant triangulation of $\widetilde{M}.$ Since each edge is essential, every tetrahedron in $\widetilde{M}$ is embedded.

Let $Z \in \para(M).$ Then a continuous map $$D_Z \co \widetilde{M} \to \overline{\projectivedisc}^3$$ can be defined inductively as follows.

Pick an oriented 3--simplex, say $\sigma = [e_0, e_1, e_2, e_3]$ in $\widetilde{M}.$ It inherits a well-defined shape parameter $z_\sigma$ from $p(\sigma),$ and hence well-defined edge invariants.
Choose $v_0, v_1, v_2, v_3\in \partial \projectivedisc^3$ such that the cross ratio $z_{ij} = (v_i,v_j;v_k,v_l)$ agrees with the edge invariant of $\sigma$ at edge $[e_i, e_j].$ Then define $D_Z(\sigma)$ as the composition of the identification $\sigma=\Delta^3$ with the straightening map $\sigma_{v_0, v_1, v_2, v_3}.$

Now suppose $D_Z$ is defined on a triangulated subset $W$ of $\widetilde{M}.$ Let $\sigma^3$ be a 3--simplex in $\widetilde{M}$ which shares at least a 2--simplex, say $\sigma^2,$ with $W.$ Suppose $\sigma^3=[w_0,w_1,w_2,w_3],$ and
$\sigma^2 = [w_0,w_1,w_2].$ Then define $D_Z(w_3) \in \partial \projectivedisc^3$ such that the cross ratio
$$(D_Z(w_0), D_Z(w_1); D_Z(w_2), D_Z(w_3))$$
equals the edge invariant of $\sigma^3$ at $[w_0,w_1].$ Even if $w_3 \in W,$ this is well-defined since the hyperbolic gluing equations are satisfied.

Then define the extension of $D_Z$ to $W \cup \sigma$ by letting $D_Z(\sigma)$ be the composition of the identification $\sigma=\Delta^3$ with the straightening map $$\sigma_{D_Z(w_0), D_Z(w_1), D_Z(w_2), D_Z(w_3)}.$$ This is well defined since all maps are straight and the hyperbolic gluing equations are satisfied. This completes the definition of $D_Z.$
Notice that by construction, we have
$$D_Z (\widetilde{M} \setminus \widetilde{M}^{(0)} ) \subset \projectivedisc^3,$$
so only the vertices are mapped to the sphere at infinity.

There is a natural isomorphism $\pi_1(M \setminus M^{(0)}) \cong \pi_1(M)$ by the Seifert-Van Kampen theorem. For each $\gamma \in \pi_1(M),$ there is a unique element $\rho_Z(\gamma)\in \PSL$ such that
$$D_Z(\gamma x) = \rho_Z(\gamma) D_Z(x)$$
for all $x \in \widetilde{M}.$
To see this, define $\rho_Z(\gamma)$ to be the isometry which maps $D_Z(\sigma)$ to $D_Z(\gamma\sigma)$ for any 3-simplex in $\widetilde{M}.$ This is well-defined since the hyperbolic gluing equations are satisfied. We therefore have an associated representation $\rho_Z\co \pi_1(M)\to \PSL.$ This representation is uniquely determined by the map $D_Z.$ The only choice in the construction of $D_Z$ is the initial placement of a 3--simplex, and it is easy to see that a different choice results in a representation which is conjugate to $\rho_Z$ by an orientation preserving isometry of $\mathbb{H}^3.$


\subsection{Representation volume}

Given the closed 3--manifold $M$ and any representation $\rho\co \pi_1(M) \to \PSL,$ the volume of $\rho$ is defined as follows (see Dunfield \cite{Du} for details). Choose any piecewise smooth $\rho$--equivariant map $f\co \widetilde{M} \to \mathbb{H}^3.$ The form $f^*(\vol_{\mathbb{H}^3})$ descends to a form on $N.$ The volume of $\rho$ is the value of the integral of this form over $M$:
$$\vol(\rho) = \int_M f^*(\vol_{\mathbb{H}^3}).$$
The volume is independent of $f$ as any two such maps are equivariantly homotopic by a straight line homotopy (cf.\thinspace the proof of Proposition~\ref{pro:spin}). The above is a slight modification of Dunfield's definition in that he takes the absolute value of the integral, whilst we maintain dependence on the orientation of $M.$

With the notation of the previous subsection, we have the following result:

\begin{lemma}\label{lem:rep vol}
Let $M$ be a closed, oriented, triangulated $3$--manifold with the property that all edges in $M$ are essential.
Then $\vol(\rho_Z) = \vol(Z)$ for each $Z \in \para(M).$
\end{lemma}

The proof of the lemma is given in Section~\ref{sec:proofs} below. Dunfield \cite{Du} proves the following rigidity result for representation volume, which he attributes to Thurston, Gromov and Goldman:

\begin{theorem}[Thurston-Gromov-Goldman]\label{Dun}
If $M$ is a compact hyperbolic 3--manifold, and
$\rho\co\pi_1(M)\rightarrow\PSL$ a
representation with $\vol(\rho)=\vol(M),$ then $\rho$
is discrete and faithful.
\end{theorem}

Dunfield \cite{Du} in fact proves slightly more (see also \cite{FK} for a further generalisation), namely:

\begin{proposition}\label{dun} Let $M$ be a compact hyperbolic $3$-manifold, and
$\rho\co\pi_1(M)\rightarrow\PSL$ be a
representation of the fundamental group of $M.$ Then $-\vol(M) \le \vol(\rho) \le \vol(M).$ Moreover, if $\vol(\rho)=\pm \vol(M),$
then $\rho$ is discrete and faithful.
\end{proposition}

Putting Lemma~\ref{lem:rep vol} and Proposition~\ref{dun} together, we have:

\begin{corollary}\label{cor:shape vol}
Let $M$ be a closed, oriented, triangulated, hyperbolic $3$--manifold with the property that all edges in $M$ are essential. Then $-\vol (M) \le \vol(Z) \le \vol(M)$ for each $Z \in \para(M),$ and if $\vol(Z) = \pm \vol(M),$ then $\rho_Z$ is discrete and faithful.
\end{corollary}

If $M$ is not hyperbolic, the range of $\vol$ depends on both the pieces of the JSJ decomposition of $M$ and the way they glue up. The proof of Theorem 1.3 in \cite{Fr} implies the following:

\begin{lemma}[Francaviglia \cite{Fr}]\label{lem:vol of sSFS}
Suppose the closed, orientable 3--manifold $M$ is a graph manifold. Then for any representation $\rho\co \pi_1(M) \to \PSL,$ we have $\vol(\rho)=0.$
\end{lemma}


\section{Proofs}
\label{sec:proofs}


\subsection{Proof of Lemma \ref{lem:rep vol}}

Let $M$ be a closed, oriented, triangulated $3$--manifold with the property that all edges in $M$ are essential; that is, we have dropped the hypothesis that $M$ be hyperbolic. We need to show that $\vol(\rho_Z) = \vol(Z)$ for each $Z \in \para(M).$

Denote $p\co\widetilde{M}\rightarrow M$ be the universal cover of $M,$ and use the set-up from Subsection~\ref{subsec:Yoshida}.

To begin with, construct any $\rho_Z$--equivariant map $F_0\co \widetilde{M} \to \projectivedisc^3$ with the property that every standard 3--simplex in $\widetilde{M}$ is mapped by a straight map to $\projectivedisc^3.$ Any such map
can be constructed by first choosing a representative for each $\pi_1(M)$--orbit  of vertices in $\widetilde{M},$ as well as an image point in $\projectivedisc^3$ for each orbit representative. One then extends the map over the whole 0--skeleton $\rho_Z$--equivariantly, and over the 3--skeleton by straight maps. This is clearly well-defined given our rigid set-up.

Let $\sigma^0$ be a vertex in $\widetilde{M},$ and suppose $D_Z(\sigma^0) = v.$ Then let $l_{\sigma^0}$ be the geodesic ray from $F_0(\sigma^0)$ to $v.$ This gives a set of geodesic rays, one for each vertex in $\widetilde{M}.$ Since both maps, $F_0$ and $D_Z,$ are $\rho_Z$--equivariant, the set of rays is also $\rho_Z$--equivariant.

For each $t \in (0, \infty),$ define the map $F_t\co \widetilde{M} \to \projectivedisc^3$ as follows. If $\sigma^0$ is a vertex in $\widetilde{M},$ then let $F_t(\sigma^0)$ be the point on the ray $l_{\sigma^0}$ which is distance $t$ from $F_0(\sigma^0).$ Then extend $F_t$ to the 3--simplices by straight maps. Since the action of $\pi_1(M)$ on $\projectivedisc^3$ via $\rho_Z$ is by isometries, this is a well-defined, $\rho_Z$--equivariant map.

Since $D_Z$ maps every 3--simplex to an ideal hyperbolic 3--simplex (possibly flat), we are in the situation of Lemma~\ref{tech}, and it follows that
$$\lim_{t\to \infty} F_t = D_Z.$$
For each $\sigma\in \Sigma^3,$ choose a standard 3--simplex $\widetilde{\sigma} \subset p^{-1}(\sigma).$
We have:
\begin{equation}\label{3}
\begin{split}
\vol(Z)=&\sum_{\sigma\in \Sigma^3}\vol(z_{\sigma})\\
=&\sum_{\sigma\in\Sigma^3}\vol_{\mathbb{H}^3}(D_Z(\widetilde{\sigma}))\\
=&\int_{M\setminus M^{(0)}}(D_Z)^*(d\vol_{\mathbb{H}^3})\\
=&\int_M(\lim_{t\rightarrow\infty}F_t)^*(d\vol_{\mathbb{H}^3})\\
=&\lim_{t\rightarrow\infty}\int_M(F_{t})^*(d\vol_{\mathbb{H}^3})\\
=&\vol(\rho_Z),
\end{split}
\end{equation}
since each term in the limit is constant.\qed


\subsection{Proof of Theorem~\ref{thm:main}}
\label{sec:existence}


Let $M$ be a closed, oriented, triangulated, hyperbolic $3$--manifold with the property that all edges in $M$ are essential. It follows from Corollary~\ref{cor:shape vol} that whenever $\vol(Z) = \pm \vol(M),$ then $\rho_Z$ is discrete and faithful. In this case, Mostow rigidity implies that $M$ is isometric with $\mathbb{H}^3/\rho_Z(\pi_1(M)),$ where the isometry is orientation preserving if $\vol(Z) = \vol(M),$ and orientation reversing if $\vol(Z) = - \vol(M).$ It remains to prove the existence of a maximum volume solution.

Let $p\co\widetilde{M}\rightarrow M$ be the universal cover of $M.$
For each $i\in\{1,..,|V|\},$ choose a geodesic $l_i$ passing through
$v_i$ so that the hypotheses in the proof of
Proposition~\ref{pro:spin} are satisfied, namely, for every
$3$-simplex $\widetilde{\sigma}$ in $\widetilde{M}$ any two lifts of
geodesics in $L=\{l_i\}$ which pass through different vertices of
$\widetilde{\sigma}$ have no endpoints in common. A generic choice
of $L$ satisfies the requirement. Indeed, pick a lift
$\widetilde{v_i}$ for each $v_i$, and select
$v_i^*\in\partial\mathbb{H}^3$. Consider $\widetilde{l_i}$ to be the
geodesic from $\widetilde{v_i}$ to $v_i^*$ and let $L$ the the image
of $\{\widetilde{l_i}\}$. For generic choices of $\{v_i^*\}$,
$\gamma\cdot\widetilde{l_i}$ and $\widetilde{l_j}$ do not have an
end point in common.

Let $\sigma$ be a $3$-simplex in $M,$ and $\widetilde{\sigma}$ a
lift of $\sigma$ in $\widetilde{M}.$ By (4) of
Proposition~\ref{pro:spin},
$\widetilde{\sigma}^*\doteq\widetilde{F}_{\infty}(\widetilde{\sigma})$
is a non-degenerate ideal tetrahedron in
$\overline{\projectivedisc^3},$ and has the associated shape
parameters $Z_{\widetilde{\sigma}^*}=(z_{q_i}),$ $i=1,2,3.$ For any
$3$--simplex $\sigma$ in $M,$ we assign the shape parameters
$z_{q_i}$ of $\widetilde{\sigma}^*$ to the corresponding
$q_i\subset\sigma,$ and get an assignment
$Z_{\infty}=(z_{\sigma})\in(\mathbb{C}\setminus\{0,1\})^Q.$ We claim
that $Z_{\infty}$ is a solution to hyperbolic gluing equation, i.e.,
$Z_{\infty}$ satisfies (a) and (b) of Definition \ref{ThGE}.

Indeed, by (2) of Proposition \ref{wellknown}, (b) is obvious. To
show (a),  let $e$ be an edge in $\tri,$ and $\widetilde{e}$ a lift
of $e$ in $\widetilde{\tri}$ with end points $\widetilde{u}$ and
$\widetilde{w}.$ Let
$\widetilde{\sigma}_1,.$..,$\widetilde{\sigma}_k$ be the tetrahedra
in $\widetilde{\tri}$ having $\widetilde{e}$ as an edge in cyclic
order and $\widetilde{q}_i$ be the normal quadrilateral in
$\widetilde{\sigma}_i$ facing $\widetilde{e}.$ Let $\widetilde{v}_i$
and $\widetilde{v}_{i+1}$ be the other two vertices of
$\widetilde{\sigma}_i$ so that
$\widetilde{v}_i\in\widetilde{\sigma}_{i-1}\cap\widetilde{\sigma}_i.$
By (4) of Proposition~\ref{pro:spin}, we get $k$ ideal tetrahedra
$\widetilde{\sigma}_{i}^*=\widetilde{F}_{\infty}(\widetilde{\sigma}_i)$
sharing the geodesic
$\widetilde{e}^*=\widetilde{F}_{\infty}(\widetilde{e})$ as an edge.

Without loss of generality, we may assume that $\widetilde{e}^*$ is the geodesic from $0$ to $\infty.$ Suppose $z_i$ is the complex number assigned to the normal quadrilateral $q_i$ in $\sigma_i$ facing $e,$ i.e., the shape parameter of $\widetilde{\sigma}_{i}^*$ at the normal quadrilateral $\widetilde{q}_i$ , we have
\begin{align*}
\prod_{q\in Q} z_q^{i(q,e)}=&\prod_{i=1}^k z_{\widetilde{q}_i}\ \ &&\text{by definition of}\ Z_{\infty}\\
=&\prod_{i=1}^k(0,\infty;\widetilde{v}_{i}^*,\widetilde{v}_{i+1}^*)\ \ &&\text{by Definition \ref{shape}}\\
=&\prod_{i=1}^k\frac{\widetilde{v}_{i}^*}{\widetilde{v}_{i+1}^*} && \\
=&1,&&\\
\end{align*}
where $\widetilde{v}_{k+1}^*$ is understood to be $\widetilde{v}_1^*,$ and this verifies (a).

To prove the volume identity, we have the following calculation.
\begin{align*}
\vol(Z_{\infty})=&\sum_{\sigma\in \Sigma^3}\vol(z_{\sigma})&&\\
=&\sum_{\sigma\in\Sigma^3}\vol_{\mathbb{H}^3}(\widetilde{\sigma}_{\infty})\ \ &&\text{by Definition \ref{vol} and (\ref{Milnor})};\\
=&\int_{M\setminus M^{(0)}}(F_{\infty})^*(d\vol_{\mathbb{H}^3})\ \ &&\text{by (4) of Proposition \ref{pro:spin}}\\
=&\int_M(\lim_{t\rightarrow\infty}F_t)^*(d\vol_{\mathbb{H}^3})\ \ &&\text{by (3) of Proposition \ref{pro:spin}}\\
=&\lim_{t\rightarrow\infty}\int_M(F_t)^*(d\vol_{\mathbb{H}^3})\ \
&&\text{by Lemma \ref{tech}}.
\end{align*}
By (5) of Proposition \ref{pro:spin}, $F_t$ is homotopic to the
identity map of $M,$ so we have
\begin{equation*}
\lim_{t\rightarrow\infty}\int_MF_t^*(d\vol_{\mathbb{H}^3})
=\lim_{t\rightarrow\infty}\int_Md\vol_M =\vol(M).
\end{equation*}
This completes the proof.
 \qed


\subsection{Proof of Theorem~\ref{thm:app}}

One first needs to decide whether $M$ is irreducible and atoroidal. The fact that this can be done follows from the work of Haken; see Jaco and Tollefson \cite{JT} for a complete exposition.

So we may suppose that $M$ is a triangulated, closed, irreducible, atoroidal and oriented 3--manifold. It follows from Thurston's Geometrisation Conjecture (which now is a complete theorem due to Perelman \emph{et al.}\thinspace \cite{perelman1, perelman2, perelman3, morgan-tian, morgan-tian2, kleiner-lott, cao-zhu}) that $M$ is either hyperbolic or a small Seifert fibred space. By passing to a barycentric sub-division, we may assume that all edges are essential.

The parameter space has finitely many Zariski components and the volume function is constant on each component. (In fact, it is constant on topological components.) Hence pick one point from each component. If $\vol$ vanishes on each, or if the parameter space is empty, then it follows from Theorem \ref{thm:main} that $M$ is not hyperbolic. Whence it is a small Seifert fibred space. In this case, Sela \cite{Se}, Section 10, describes an effective algorithm to recognise the Seifert fibred structure using the fundamental group, and the classification of Seifert fibred spaces can be used to construct the structure.

Hence suppose that there is a point $Z$ with $\vol(Z)\neq 0.$ It follows from Lemma~\ref{lem:vol of sSFS} that $M$ cannot be a small Seifert fibred space and hence is hyperbolic. Theorem~\ref{thm:main} implies that the values of $\vol$ lie in $[-\vol(M), \vol(M)]$ and that both bounds are attained. Then a discrete and faithful representation of $\pi_1(M)$ into $\PSL$ is the holonomy representation determined by a point of maximal volume. The structure can now be constructed as described by Manning (see \cite{Ma}, Section 4).

It remains to address how to turn the above outline into a rigorous algorithm. The first part, checking that $M$ is irreducible and atoroidal using normal surface theory, is clearly rigorous and the routines are implemented in Regina \cite{Bu}.

The algorithms from algebraic geometry which are needed involve computations with algebraic numbers over the rationals. Manning (see \cite{Ma}, Section 2) has summarised most of what we need using results from \cite{BW} and \cite{Lo}; including prime decomposition and determining the dimension of an ideal. The ability to pick a point from each Zariski component of positive dimension using exact arithmetic follows from the fact that in order to determine dimension, one finds a maximally independent set of coordinates. Each is a non-constant function on the variety, and one can therefore use standard elimination theory (see for instance \cite{CLO}) to determine a point on the Zariski component algorithmically.

Having one point from each Zariski component, say $\{ Z_1, \ldots, Z_k\},$ it suffices to compute $\vol(Z_i)$ up to high enough precision in order to determine the maximum point as follows. First compute $\vol(Z_i)$ up to a pre-determined precision. If each $\vol(Z_i)$ is estimated to be less than $0.9,$ then the maximum must be equal to zero, since recent work of Gabai, Meyerhoff and Milley \cite{GMM, Mil} has shown that the Weeks manifold (which has volume approximately 0.9427) is the orientable hyperbolic 3--manifold of smallest volume. Otherwise, compare the values to determine the maximum. In the event that two or more values are equal at the pre-determined precision, one can compute the characters of the associated representations in exact arithmetic. If they agree, either of the points will be the desired maximum. If they don't, the procedure of incrementally increasing the precision will eventually terminate. This algorithm is to a point  not only theoretical, but also practical: computing volume up to any given precision (but subject to hardware limitations) is implemented in the software snap \cite{snap-paper}. \qed


\section*{Acknowledgements}

Research of the first and the third author is partially supported by
the NSF.

Research of the second author is partially funded by a UQ New Staff Research Start-Up Fund and under the Australian Research Council's Discovery funding scheme (DP1095760).

The third author would like to thank Ren Guo, Jingzhou Sun and Yuan Yuan for helpful discussions.

\bibliographystyle{unsrt}
\bibliography{ref}




\address{Feng Luo,\\ Department of Mathematics,\\ Rutgers University,\\ New Brunswick, NJ 08854, USA\\
(fluo@math.rutgers.edu)\\--}

\address{Stephan Tillmann,\\ School of Mathematics and Physics,\\ The University of Queensland,\\ Brisbane, QLD 4072, Australia\\
(tillmann@maths.uq.edu.au)\\--}

\address{Tian Yang,\\ Department of Mathematics,\\ Rutgers University,\\ New Brunswick, NJ 08854, USA\\
(tianyang@math.rutgers.edu)}

\Addresses

\end{document}